\documentclass{amsart}
\usepackage{amsthm, amsmath, amssymb, color}

\newtheorem{theorem}{Theorem}[section]
\newtheorem{lemma}[theorem]{Lemma}
\newtheorem{prop}[theorem]{Proposition}
\newtheorem{cor}[theorem]{Corollary}

\theoremstyle{definition}

\newtheorem{example}[theorem]{Example}

\theoremstyle{remark}
\newtheorem{remark}[theorem]{Remark}

\numberwithin{equation}{section}

\DeclareMathOperator{\Supp}{Supp}

\begin{document}

\title{Gaps in Taylor Series of Algebraic Functions}

\author{Seth Dutter}
\address{University of Wisconsin - Stout\\
Menomonie, WI}
\curraddr{}
\email{dutters@uwstout.edu}
\thanks{}

\date{October 24, 2014}



\begin{abstract}
Let $f$ be a rational function on an algebraic curve over the complex numbers. For a point $p$ and local parameter $x$ we can consider the Taylor series for $f$ in the variable $x$. In this paper we give an upper bound on the frequency with which the terms in the Taylor series have $0$ as their coefficient.
\end{abstract}

\maketitle

\section{Introduction}\label{section.intro}

Let $f$ be an algebraic function over the complex numbers and suppose that in local coordinates $f$ can be represented by a Taylor series of the form
\[
f = \sum_{n=0}^\infty \alpha_n x^{a_n},
\]
where the $a_n$ are a strictly increasing sequence of non-negative integers and the $\alpha_n$ are nonzero complex numbers. Our goal is to understand the possible gaps of the $a_n$.

If $f=p(x)/q(x)$ is rational, the gap sequences are fully understood. For sufficiently large $n$ the coefficients of the Taylor series of $f$ satisfy a linear homogeneous recurrence relation of degree $\deg(q(x))$. Consequently, the length of the gaps in the Taylor series are eventually bounded by $\deg(q(x))-1$. An even stronger conclusion is given by the Skolem-Mahler-Lech Theorem \cite{mahler56}, which states that the zero terms of the Taylor series of a rational function consist of finitely many arithmetic progressions plus a finite set. Therefore $\lim_{n\rightarrow\infty}a_n/n$ exists and is effectively bounded.

In \cite{Christol1980}, Christol, Kamae, Mend\`{e}s France, and Rauzy prove an equivalency between formal algebraic power series over finite fields and automatic sequences. An overview of this theorem and some of its generalizations can be found in \cite{denef87}. Less has been established for algebraic functions over the complex numbers, but there are some classic results from the field of complex analysis. For instance, the Fabry gap theorem \cite{fabry96} states that if
\[
\lim_{n\rightarrow\infty} \frac{a_n}{n} = \infty
\]
then $f$ cannot be analytically continued to any point beyond the domain of convergence. In particular $f$ is not algebraic.

We will show that for any algebraic function $\limsup_{n\rightarrow \infty} a_n/n$ is effectively bounded by the number of distinct poles of $f$ plus a term corresponding to the choice of local coordinates. This is a consequence of Theorem \ref{theorem.main} which provides a bound for each $a_n$ in terms of similar data.

\section{Main Theorem}

Throughout this section we will let $C$ be a smooth complete algebraic curve of genus $g$ over $\mathbb{C}$ and $K$ be the function field of $C$. At each point $q\in C$ we will denote by $v_q$ the unique surjective discrete valuation at $q$, $v_q: K^\times\rightarrow \mathbb{Z}$. Recall that for $f\in K^\times$ the height of $f$ is defined to be
\[
h(f) = -\sum_{q\in C}\min\{v_q(f), 0\}
\]
and has the property that $v_q(f) \leq h(f)$ for any $q\in C$. We now state our main theorem.

\begin{theorem}[Main Theorem]\label{theorem.main}
Let $f\in K$ be a rational function and $x\in K$ a local parameter at $p\in C$ such that $f\notin \mathbb{C}[x]$ and $v_p(f) = 0$. In the completion of the local ring at $p$ we can uniquely write
\[
f = \sum_{n=0}^\infty \alpha_n x^{a_n},
\]
where $\alpha_n \neq 0$ for all $n$ and $\{a_n\}$ is a sequence of strictly increasing integers with $a_0 = 0$. Then for all positive integers $n$
\begin{equation}\label{inequality.main}
a_n \leq h(f)+(n-1)\left(\#S_1 + \sum_{q\in S_2}\left(v_q\left(\frac{dx}{dx_q}\right)+1\right)\right),
\end{equation}
where
\begin{align*}
S_1 &= \{q\in C : v_q(f) < 0 \text{ and } v_q(dx/dx_q) = 0\},\\
S_2 &= \{q\in C : v_q(x)=0\text{ and } v_q(dx/dx_q) > 0\},
\end{align*}
and $x_q$ is a local parameter at $q$.
\end{theorem}

As an immediate consequence of the above theorem we get an effective bound on $\limsup_{n\rightarrow \infty} a_n/n$. Namely we have
\[
\limsup_{n\rightarrow \infty} \frac{a_n}{n} \leq \#S_1 + \sum_{q\in S_2}\left(v_q\left(\frac{dx}{dx_q}\right)+1\right).
\]
Since $S_1$ is a subset of the set of poles of $f$, the limit superior is bounded by the number of poles of $f$ plus a term corresponding to the choice of local coordinates.

In the case that $C$ has genus $0$ Theorem \ref{theorem.main} is sharp.

\begin{example} Let $C = \mathbb{P}^1_\mathbb{C}$ with $K = \mathbb{C}(x)$ and $x$ our local parameter. Under this hypothesis $S_2$ is empty and Equation \ref{inequality.main} simplifies to
\[
a_n \leq h(f) + (n-1)\#S_1.
\]
For each pair of positive integers $k, m$ with $k > m$ we define the rational function
\[
f = 1+\frac{x^{k}}{1-x^m} = 1 + \sum_{n=1}^\infty x^{(n-1)m+k}.
\]
By construction $a_n = k+(n-1)m$ for all $n\geq 1$. However, $\#S_1 = m$ and $h(f) = k$, therefore $a_n = h(f) + (n-1)\#S_1$ as desired.
\end{example}

In the case that the genus is positive, it is not clear whether Theorem \ref{theorem.main} is sharp. No examples demonstrating sharpness are known to the author. In particular the term corresponding to the choice of local coordinates may be too large.

To prove Theorem \ref{theorem.main} we begin by constructing auxiliary rational functions that vanish to order $a_n$ for each positive integer $n$.

\begin{lemma}\label{lemma.linearcombination}
Let $x, f \in K$ and $p \in C$ be as in Theorem \ref{theorem.main}. Then for every positive integer $n$ there exist constants $c_i\in \mathbb{C}$, for $i=0, \ldots, n$, such that
\[
v_p\left(c_0 + c_1 f + c_2 x f' + c_3 x^2 f''+ \cdots + c_n x^{n-1} f^{(n-1)}\right) = a_n,
\]
where all derivatives are taken with respect to $x$.
\end{lemma}

\begin{proof}
We begin by considering the $(n+1)\times (n+1)$ matrix
\[
A  = \begin{pmatrix}
1 & \alpha_0 & 0 & \cdots & 0\\
0 & \alpha_1 x^{a_1} & \alpha_1 (a_1)_1 x^{a_1} & \cdots & \alpha_1 (a_1)_{n-1} x^{a_1}\\
0 & \alpha_2 x^{a_2} & \alpha_2 (a_2)_1 x^{a_2} & \cdots & \alpha_2 (a_2)_{n-1} x^{a_2}\\
\vdots & \vdots & \vdots & \ddots & \vdots\\
0 & \alpha_n x^{a_n} & \alpha_n (a_n)_1 x^{a_n} & \cdots & \alpha_n (a_n)_{n-1} x^{a_n}
\end{pmatrix},
\]
where the $i$-th column, for $i=2, \ldots, n+1$, is the first $n+1$ monomials of the power series expansion of $x^{i-1} f^{(i-1)}$. Similarly, we define $B$ to be the $n\times (n+1)$ matrix consisting of the first $n$ rows of $A$. Since $B$ represents an underdetermined system there exists some nonzero vector $c$ such that $Bc = 0$. We claim that the components of $c$ are the $c_i$ in the conclusion of our theorem. All that remains to be shown is that $Ac \neq 0$ guaranteeing that the $x^{a_n}$ term does not vanish in this sum. Let $A_{1,1}$ be the submatrix of $A$ after removing the first row and column. It suffices to show that $\det(A_{1,1}) \neq 0$. After factoring out $x^{i-1}$ from the $i$-th column of $A_{1,1}$, for $i=1, \ldots, n$, we have
\[
\det(A_{1,1}) = x^{n(n-1)/2} \det \begin{pmatrix}
\alpha_1 x^{a_1} & \alpha_1 (a_1)_1 x^{a_1-1} & \cdots & \alpha_1 (a_1)_{n-1} x^{a_1-(n-1)}\\
\alpha_2 x^{a_2} & \alpha_2 (a_2)_1 x^{a_2-1} & \cdots & \alpha_2 (a_2)_{n-1} x^{a_2-(n-1)}\\
\vdots & \vdots & \ddots & \vdots\\
\alpha_n x^{a_n} & \alpha_n (a_n)_1 x^{a_n-1} & \cdots & \alpha_n (a_n)_{n-1} x^{a_n-(n-1)}
\end{pmatrix}.
\]
However, the latter determinant is the Wronskian $W(\alpha_1 x^{a_1},\alpha_2 x^{a_2}, \ldots, \alpha_n x^{a_n})$. Since the entries of this Wronskian are monomials of strictly increasing degree they must be linearly independent. Therefore we conclude that $\det(A_{1,1})\neq 0$ and $Ac\neq 0$.
\end{proof}

We now proceed to bound from below the valuation of the derivatives of $f$ with respect to $x$.

\begin{prop}\label{prop.valuations}
Let $f, x\in K$ be nonconstant and $q\in C$ be arbitrary. Then for any non-negative integer $n$
\[
v_q\left(\frac{d^nf}{dx^n}\right) \geq v_q(f)-n\left(v_q\left(\frac{dx}{dx_q}\right)+1\right),
\]
where $x_q$ is a local parameter at $q$.
\end{prop}

\begin{proof}
If $d^nf/dx^n$ is identically $0$ then the inequality trivially holds. Without further mention we will assume that all valuations are finite. We proceed by induction. If $n=0$ the result is immediate. Suppose now that $v_q\left(d^nf/dx^n\right) \geq v_q(f)-n(v_q(dx/dx_q)+1)$ for some $n\geq 0$. By the chain rule
\[
v_q\left(\frac{d^{n+1}f}{dx^{n+1}}\right) = v_q\left(\frac{d}{dx_q}\frac{df^n}{dx^n}\right)-v_q\left(\frac{dx}{dx_q}\right).
\]
Since $x_q$ is a local parameter at $q$, differentiating with respect to $x_q$ drops the order by at most $1$. Combining this observation with our induction hypothesis gives the inequality
\begin{align*}
v_q\left(\frac{d^{n+1}f}{dx^{n+1}}\right) &\geq \left(v_q(f)-n\left(v_q\left(\frac{dx}{dx_q}\right)+1\right)-1\right)-v_q\left(\frac{dx}{dx_q}\right)\\
&= v_q(f)-(n+1)\left(v_q\left(\frac{dx}{dx_q}\right)+1\right),
\end{align*}
which, by the principle of mathematical induction, completes the proof.
\end{proof}

\begin{remark}\label{remark.insupport}
If $v_q(x) \neq 0$, then  $v_q(dx/dx_q) = v_q(x)-1$. Under this additional hypothesis Proposition \ref{prop.valuations} simplifies to $v_q(d^nf/dx^n) \geq v_q(f)-nv_q(x)$.
\end{remark}

We now proceed to prove Theorem \ref{theorem.main}.

\begin{proof}[Proof of Theorem \ref{theorem.main}]
Let $n$ be a positive integer, by Lemma \ref{lemma.linearcombination} we can construct a rational function
\[
F = c_0 + c_1 f + c_2 x f' + c_3 x^2 f''+ \cdots + c_n x^{n-1} f^{(n-1)}
\]
such that $v_p(F) = a_n$. Since $a_n \leq h(F)$ it suffices to bound $h(F)$. To do so we consider four cases.

\emph{Case 1.} Let $q\in S_1 = \{q\in C : v_q(f) < 0 \text{ and } v_q(dx/dx_q) = 0\}$. It follows that $v_q(x) \geq 0$, and by Proposition \ref{prop.valuations}, for each $i=0,\ldots, n-1$, we have
\[
v_q\left(x^i\frac{d^if}{dx^i}\right) \geq v_q\left(\frac{d^if}{dx^i}\right) \geq v_q(f)-i.
\]
We next bound the height of $F$ over the set $S_1$.
\[
-\sum_{q\in S_1} \min\{v_q(F), 0\} \leq -\sum_{q\in S_1} \min\{v_q(f), 0\} + (n-1)\#S_1.
\]

\emph{Case 2.} Let $q\in S_2 = \{q \in C : v_q(dx/dx_q) > 0 \text{ and } v_q(x) = 0\}$. Then by Proposition \ref{prop.valuations}, for each $i=0,\ldots, n-1$,
\[
v_q\left(x^i\frac{d^if}{dx^i}\right) \geq v_q(f)-i\left(v_q\left(\frac{dx}{dx_q}\right)+1\right).
\]
The above inequality gives us the following bound on the height of $F$ over $S_2$.
\[
-\sum_{q\in S_2} \min\{v_q(F), 0\} \leq -\sum_{q\in S_2} \min\{v_q(f), 0\} + (n-1)\sum_{q\in S_2}\left(v_q\left(\frac{dx}{dx_q}\right)+1\right).
\]

\emph{Case 3.} Let $q\in S_3 = \{q\in C : v_q(f) \geq 0 \text{ and } v_q(dx/dx_q) = 0\}$. By Proposition \ref{prop.valuations}, for each $i=0,\ldots, n-1$,
\[
v_q\left(x^i\frac{d^if}{dx^i}\right) \geq 0.
\]
Therefore the height of $F$ over $S_3$ is $0$.

\emph{Case 4.} Let $q\in S_4 = \{q\in C : v_q(dx/dx_q) \neq 0 \text{ and } v_q(x) \neq 0\}$. By Remark \ref{remark.insupport}, for each $i=0, \ldots, n-1$,
\[
v_q\left(x^i\frac{d^if}{dx^i}\right) \geq v_q(f).
\]
Lastly we bound the height of $F$ over $S_4$.
\[
-\sum_{q\in S_4} \min\{v_q(F), 0\} \leq -\sum_{q\in S_4} \min\{v_q(f), 0\}.
\]

Combining the four cases we have
\begin{multline*}
-\sum_{q\in C}\min\{v_q(F), 0\} \leq -\sum_{q\in C} \min\{v_q(f), 0 \}\\
+(n-1)\#S_1+(n-1)\sum_{q\in S_2} \left(v_q\left(\frac{dx}{dx_q}\right)+1\right).
\end{multline*}
Therefore the desired inequality holds.
\end{proof}

The summation over $S_2$ in the conclusion of our theorem can be bounded above in terms of the genus and the support of $x$. In order to do so we need the following lemma.

\begin{lemma}\label{lemma.rr}
For any nonconstant $x\in K$
\[
\sum_{q\notin\Supp\{x\}} v_q\left(\frac{dx}{dx_q}\right) = \#\Supp\{x\}+2g-2.
\]
\end{lemma}

\begin{proof}
By Riemann-Roch
\begin{align*}
2g-2 &= \sum_{q\in C} v_q\left(\frac{1}{x}\frac{dx}{dx_q}\right)\\
&= \sum_{q\in \Supp\{x\}}v_q\left(\frac{1}{x}\frac{dx}{dx_q}\right)+\sum_{q\notin\Supp\{x\}}v_q\left(\frac{dx}{dx_q}\right)\\
&= \sum_{q\in \Supp\{x\}}-1+\sum_{q\notin\Supp\{x\}}v_q\left(\frac{dx}{dx_q}\right),
\end{align*}
which, after rearranging terms, completes the proof.
\end{proof}

\begin{cor}\label{cor.simple}
Let $f, x\in K$ and $p\in C$ be as in the statement of Theorem \ref{theorem.main}. Then for all positive integers $n$
\[
a_n \leq h(f)+(n-1)\left(\#S_1+2(\#\Supp\{x\} +2g-2)\right).
\]
\end{cor}

\begin{proof}
Recall that $S_2 = \{q \in C : v_q(x) = 0 \text{ and } v_q(dx/dx_q) > 0\}$. Then we have
\begin{align*}
\sum_{q\in S_2} \left(v_q\left(\frac{dx}{dx_q}\right)+1\right) &\leq 2\sum_{q\in S_2}v_q\left(\frac{dx}{dx_q}\right)\\
&=2\sum_{q\notin \Supp\{x\}} v_q \left(\frac{dx}{dx_q}\right).
\end{align*}
By Lemma \ref{lemma.rr} the right hand side is precisely $2(\#\Supp\{x\}+2g-2)$. Substituting this value into Equation \ref{inequality.main} gives the desired result.
\end{proof}

\bibliographystyle{abbrv}
\bibliography{sparsealgebraic}

\end{document}